\documentclass[12pt, a4paper]{amsart}
\usepackage[utf8]{inputenc}
\usepackage[T1]{fontenc}
\usepackage[english]{babel}
\usepackage{amsmath,amsfonts,amssymb,latexsym,amscd}
\usepackage{cite}
\theoremstyle{plain}
\newtheorem{theorem}{Theorem}
\newtheorem{proposition}{Proposition}
\newtheorem{corollary}{Corollary}
\theoremstyle{definition}
\newtheorem{definition}{Definition}

\usepackage{hyperref}

\begin{document}

\title{Bi-Equivariant Extensions Of  Maps} 

\author{Pavel S. Gevorgyan}

\address{Moscow State Pedagogical University}

\email{pgev@yandex.ru}



\begin{abstract}
The problem of bi-equivariant extension of continuous maps of binary $G$-spaces is considered. The concept of a structural map of distributive binary $G$-spaces is introduced, and a theorem on the bi-equivariant extension of structural maps is proven. A theorem on the bi-equivariant extension of continuous maps defined on the cross sections of a distributive binary $G$-space is also proven.
\end{abstract}

\maketitle


\section{Introduction}

The notion of a binary $G$-space or a group of binary transformations of a topological space was introduced in the paper \cite{Gev}. This concept in algebra was considered and studied in \cite{movsisyan}. When a group $G$ acts binarily on a topological space $X$, then there exists a homomorphism of the group $G$ into the group $C_2^*(X)$ of all invertible continuous binary operations of $X$ with the identity element $e(x,x')=x'$ and the composition of two binary operations $f,g\in C_2^*(X)$, defined by the formula 
$$(f\circ g)(x,x')=f(x,g(x,x')),$$ 
where $x,x'\in X$. Consequently, the elements of the group $G$ can be represented as invertible continuous binary operations of the space $X$.

The binary action $\alpha$ of the group $G$ on the space $X$ generates a continuous family of ordinary actions $\{\alpha_x, \ x\in X\}$ of the group $G$ on $X$. Binary $G$-spaces and bi-equivariant maps form a category, which is a natural extension of the category of $G$-spaces and equivariant maps.

When transferring and studying the basic concepts and results of the theory of $G$-spaces to the theory of binary $G$-spaces, natural difficulties and significant differences arise. For example, the orbits of points in a binary $G$-space may intersect, and therefore, orbit spaces cannot be defined in usual terms using minimal bi-invariant subsets. However, this can be done for the so-called distributive binary $G$-spaces. These and other questions of equivariant and bi-equivariant topology can be found in papers \cite{Bredon,Dieck,Gev2,Gev,Gev-3,Gev-Naz}.

This article is devoted to the problem of bi-equivariant extension of continuous maps defined on closed subsets of binary $G$-spaces. A sufficient condition for the bi-equivariant extension of a continuous map $f:A\to Y$, where $A\subset X$ is a closed subset and $X$ and $Y$ are binary $G$-spaces, is obtained (Theorem \ref{th-1}).

In a distributive binary $G$-space $X$, the question of bi-equivariant extension is considered for maps defined on the cross sections of the projection $\pi : X \to X|G$. Since the saturation of a cross section $C \subset X$ coincides with the space $X$ itself, the map $f:C \to Y$ can be bi-equivariantly extended uniquely to the whole space $X$ (Theorem \ref{th-4}, Corollary \ref{th-4}).


\section{Preliminaries} 

Let $X$ be a topological space and let $G$ be an arbitrary topological group. A \textit{binary action} of group $G$ on $X$ is a 
continuous map $\alpha :G\times X^2\to X$ such that
\begin{equation*}\label{eq(1)}
\alpha (gh, x_1,x_2)=\alpha(g, x_1, \alpha(h,x_1,x_2)), \quad \alpha (e, x_1,x_2)=x_2,
\end{equation*}
or
$$gh(x_1,x_2)=g(x_1, h(x_1,x_2)), \quad e(x_1,x_2)=x_2$$
for all $g,h\in G$ and $x_1,x_2 \in X$, where $e$ is the identity of $G$.

By a \textit{topological binary transformation group} or \textit{binary $G$-space} we mean a triple $(G,X,\alpha)$, where $\alpha$ is a binary action of group $G$ on $X$.

There are two natural binary actions of the topological group $G$ on itself given by the following formulas:
\begin{equation}\label{eq-GonGdistr}
g(g_1,g_2) = g_1gg_1^{-1}g_2
\end{equation}
and 
\begin{equation}\label{eq-GonG}
g(g_1,g_2) = g_1^{-1}gg_1g_2,
\end{equation}
where $g,g_1,g_2\in G$.

For any subsets $A$ and $B$ of a binary $G$-space $X$, and an arbitrary $K$ in $G$, we denote
$$K(A,B)=\{g(a,b); \quad g\in K, a\in A, b\in B\}.$$
If $G$ is a compact group, and $A$ is a closed subset of a binary $G$-space $X$, then $G(A, A)$ is closed in $X$ \cite[Theorem 2]{Gev-3}.

A subset $A\subset X$ is called \emph{bi-invariant} if $G(A,A) = A$ holds. A bi-invariant subset $A \subset X$ itself constitutes a binary $G$-space with the induced binary action and is called a \emph{binary $G$-subspace} of $X$. The intersection of bi-invariant sets of a binary $G$-space $X$ is bi-invariant. However, the union of bi-invariant sets, in general, is not bi-invariant.

The minimal bi-invariant subset containing the point $x \in X$ is called the \emph{orbit} of the point $x$ and is denoted by $[x]$. It is clear that $G(x,x)\subset [x]$. Therefore, if $G(x,x)$ is a bi-invariant set, then $G(x,x) = [x]$.

Unlike $G$-spaces, in binary $G$-spaces the orbits of points may intersect \cite[Example 3]{Gev-Naz}. Consequently, the concept of an orbit space cannot be directly extended to all binary G-spaces.

A binary action of the group $G$ on $X$ is called \textit{distributive} if the following condition is satisfied:
\begin{equation}\label{eq1-1}
g(h(x,x_1), h(x,x_2))=h(x,g(x_1, x_2))
\end{equation}
for any $x, x_1, x_2 \in X$ and $g, h\in G$.
In this case, $X$ is called a \textit{distributive} binary $G$-space.

The binary action \eqref{eq-GonGdistr} of a group $G$ on itself is distributive, while the binary action defined by the formula \eqref{eq-GonG} is not necessarily distributive, in general.

In a distributive binary $G$-space $X$, the set $G(x,x)$ is bi-invariant for all $x\in X$ \cite[Proposition 1]{Gev-3}, and any two orbits either are disjoint or coincide \cite[Proposition 6]{Gev-Naz}. Hence, a distributive binary $G$-space $X$ is partitioned by its orbits, and one can define the notion of the orbit space $X|G$ in usual terms. If $G$ is a compact topological group, and $X$ is a distributive binary $G$-space, then the orbit space $X | G$ is Hausdorff \cite[Theorem 5]{Gev-3}, and the projection $\pi:X\to X|G$ possesses important properties. The following theorem can be found in \cite{Gev-3}.

\begin{theorem}\label{th-2}
Let $G$ be a compact topological group, and let $X$ be a distributive binary $G$-space. Then the projection $\pi:X\to X|G$ is

(a) a closed map;

(b) a proper map.
\end{theorem}

A continuous map $f:X\to Y$ between binary $G$-spaces $(G,X,\alpha)$ and $(G,Y,\beta)$ is called a \textit{bi-equivariant map} provided
$$f(\alpha(g,x_1,x_2))=\beta(g,f(x_1),f(x_2))$$
or
$$f(g(x_1,x_2))=g(f(x_1),f(x_2))$$
for all $g\in G$ and $x_1,x_2 \in X$.

A bi-equivariant map \(f:X\to Y\), which is also a homeomorphism, is called an \textit{equivalence} of binary \(G\)-spaces, or a \textit{bi-equimorphism}. All binary \(G\)-spaces and bi-equivariant maps form a category.

For more details on the concepts, definitions, and results presented above, one can refer to the works \cite{Gev0,Gev2,Gev1,Gev,Gev-3,Gev-Iliadis,Gev-Quitzeh,Gev-Naz}.


\section{Structural maps and their bi-equivariant extension}

Let $X$ be a binary $G$-space, and let $A$ be any subset of $X$. The minimal bi-invariant subset $\widetilde{A}$ of $X$ which contains  a set $A$ is called the \textit{saturation} of $A$. 

Let us recursively define the sets $A^n$, $n=1,2, \ldots \,$, as follows: 
$$A^1 = G(A,A), \quad  A^2 = G(A^1,A^1), \quad \ldots \,, \quad A^n = G(A^{n-1},A^{n-1}), \quad \ldots$$ 

\begin{proposition}
		For any subset $A\subset X$ of a binary $G$-space $X$,
		\[
		\widetilde{A}= \bigcup\limits_{n=1}^{\infty} A^n,
		\] 
		where $\widetilde{A}$ is the saturation of $A$.
\end{proposition}

\begin{proof}
It suffices to note that
$$A\subset A^1 \subset A^2 \subset \cdots \subset A^n \subset \cdots $$
and that $\bigcup\limits_{n=1}^{\infty} A^n\subset X$  is a bi-invariant subset. 
\end{proof}

Now, let us denote an element $x=g_1(a_1,a_2) \in A^1$ by $[g_1;a_1,a_2]$: 
$$x= [g_1;a_1,a_2].$$

Note that any element $x\in A^2$ has a form $g_1(g_2(a_1,a_2), g_3(a_3,a_4))$
which we denote by $[g_1,g_2,g_3;a_1,a_2, a_3, a_4]$:
$$x=[g_1,g_2,g_3;a_1,a_2, a_3, a_4].$$

Similarly, any element $x\in A^n$ is defined by a collection of some elements $g_1, \ldots , g_{2^{n}-1} \in G$ and $a_1, \ldots \,, a_{2^n} \in A$: 
$$x=[g_1, \ldots \,, g_{2^{n}-1};a_1, \ldots \,, a_{2^n}].$$

\begin{definition}
Let $X$ and $Y$ be binary $G$-spaces, and let $A$ be a subset of $X$. We say that a continuous map $f:A\to Y$ is a \textit{structural map} if the following conditions hold:

\vspace{3mm}

({\bf SM1})  If $a_1, a_2$ and $g(a_1,a_2) \in A$, $g\in G$, then $f(g(a_1,a_2))=g(f(a_1),f(a_2))$, 

\vspace{2mm}

({\bf SM2})  $[g_1, \ldots \,, g_{2^{n}-1};a_1, \ldots \,, a_{2^n}] = [g'_1, \ldots \,, g'_{2^{m}-1}; a'_1, \ldots \,, a'_{2^m}]$  
implies
$$[g_1, \ldots \,, g_{2^{n}-1}; f(a_1), \ldots \,, f(a_{2^n})] = [g'_1, \ldots \,, g'_{2^{m}-1}; f(a'_1), \ldots \,, f(a'_{2^m})], \vspace{1mm}$$
where $g_1, \ldots g_{2^{n}-1},  g_1', \ldots g'_{2^{m}-1} \in G$, $a_1, \ldots \,, a_{2^n}, a'_1, \ldots \,, a'_{2^m}\in A$, $n,m \in N$.
\end{definition}

It is easy to see that if $A \subset X$ is a bi-invariant subset, then any bi-equivariant map $f:A\to Y$ is a structural map.

\begin{theorem}\label{th-1}
Let $G$ be a compact group, $X$ and $Y$ be binary $G$-spaces and let $A$ be a closed subset of $X$. Then every continuous structural map $f:A\to Y$ can be extended uniquely to a continuous bi-equivariant map $\widetilde{f} : \widetilde{A} \to Y$ where $\widetilde{A}\subset X$ is the  saturation of $A$.
\end{theorem}

\begin{proof}
Let $f:A\to Y$ be a structural map. 
Consider an arbitrary element $x \in \widetilde{A}$ and assume that $x=[g_1, \ldots \,, g_{2^{n}-1};a_1, \ldots \,, a_{2^n}]$ for some $g_1, \ldots \,, g_{2^{n}-1}\in G$, $a_1, \ldots \,, a_{2^n}\in A$ and $n\in \mathbf{N}$. Now, let's define the only possible bi-equivariant extension $\widetilde{f} : \widetilde{A} \to Y$ of the map $f:A\to Y$ by the formula:
$$\widetilde{f}(x) = [g_1, \ldots \,, g_{2^{n}-1}; f(a_1), \ldots \,, f(a_{2^n})].$$
This definition is correct due to conditions ({\bf SM1}) and ({\bf SM2}). 

It is easy to note that the map $\widetilde{f} : \widetilde{A} \to Y$ is an extension of the map $f:A\to Y$. The continuity of $\widetilde{f}$ follows from the closedness of the set $A$, continuity of $f$ and the binary actions of the compact group $G$ on $X$ and $Y$.

What remains is to prove the bi-equivariance of the map $\widetilde{f}$. Let  
$$x=[g_1, \ldots \,, g_{2^{n}-1};a_1, \ldots \,, a_{2^n}] \quad \text{and} \quad x'=[g'_1, \ldots \,, g'_{2^{m}-1};a'_1, \ldots \,, a'_{2^m}]$$
$g_1, \ldots \,, g_{2^{n}-1}, g'_1, \ldots \,, g'_{2^{m}-1} \in G$, $a_1, \ldots \,, a_{2^n}, a'_1, \ldots \,, a'_{2^m} \in A$ and $n, m\in \mathbf{N}$, be two arbitrary elements of the saturation $\widetilde{A}$ of the subspace $A\subset X$, and let $g\in G$ be an any element of the group $G$. Assume that
$$g(x,x')=[g''_1, \ldots \,, g''_{2^{k}-1}; a''_1, \ldots \,, a''_{2^k}]$$
or
$$g([g_1, \ldots \,, g_{2^{n}-1};a_1, \ldots \,, a_{2^n}],[g'_1, \ldots \,, g'_{2^{m}-1};a'_1, \ldots \,, a'_{2^m}])=[g''_1, \ldots \,, g''_{2^{k}-1}; a''_1, \ldots \,, a''_{2^k}],$$
where $g''_1, \ldots \,, g''_{2^{k}-1}\in G$, $a''_1, \ldots \,, a''_{2^k}\in A$ and $k\in \mathbf{N}$.
From this equality, due to ({\bf SM2}), it follows that 
\begin{multline*}
g([g_1, \ldots \,, g_{2^{n}-1};f(a_1), \ldots \,, f(a_{2^n})], [g'_1, \ldots \,, g'_{2^{m}-1}; f(a'_1), \ldots \,, f(a'_{2^m})])= \\
= [g''_1, \ldots \,, g''_{2^{k}-1};f(a''_1), \ldots \,, f(a''_{2^k})].
\end{multline*}
Considering the last equalities, we obtain:
\begin{multline*}
\widetilde{f}(g(x,x')= \widetilde{f}([g''_1, \ldots \,, g''_{2^{k}-1}; a''_1, \ldots \,, a''_{2^k}])= [g''_1, \ldots \,, g''_{2^{k}-1};f(a''_1), \ldots \,, f(a''_{2^k})]= \\
= g([g_1, \ldots \,, g_{2^{n}-1};f(a_1), \ldots \,, f(a_{2^n})], [g'_1, \ldots \,, g'_{2^{m}-1}; f(a'_1), \ldots \,, f(a'_{2^m})])
= \\
=g(\widetilde{f}(x), \widetilde{f}(x')).
\end{multline*}
\end{proof}


\section{Cross sections and bi-equivariant extension of maps}

Let $X$ be a distributive binary $G$-space, $X|G$ be its orbit space, and $\pi: X \rightarrow X|G$ be the projection on the orbit space. 

A continuous map $\sigma:X|G \to X$ is called a {\em cross section} for $\pi:X\to X|G$ if $\pi\sigma$ is the identity on $X|G$, i.e., $\pi\sigma = 1_{X|G}$.

\begin{proposition}\label{prop-2}
Let $G$ be a compact group, $X$ a distributive binary $G$-space, and  $\sigma : X|G\to X$ a section of the projection $\pi :X\to X|G$. Then, the image of the section $\sigma(X|G)$ is a closed subset of $X$.
\end{proposition}

\begin{proof}
Let's prove that the complement $X \setminus \sigma(X|G)$ is an open set.
Let $x_0 \in X \setminus \sigma(X|G)$ be an arbitrary point. This means that the point $y_0 = \sigma\pi(x_0)$ is different from $x_0$: $y_0 \neq x_0 $. Let $U_0$ and $V_0$ be disjoint neighborhoods of the points $x_0$ and $y_0$ respectively: $U_0 \cap V_0 = \emptyset$.

Since $\sigma\pi : X \to X$ is a continuous map, then $(\sigma\pi)^{-1}(V_0)$ is an open neighborhood of the point $x_0$. Let's denote $W_0 = U_0 \cap (\sigma\pi)^{-1}(V_0)$. This set is an open neighborhood of the point $x_0$. It remains to note that $W_0 \subset X \setminus \sigma(X|G)$, i.e., for any $x \in W_0$, $x \neq \sigma\pi(x)$ holds true. Indeed, since $x \in W_0 \subset (\sigma\pi)^{-1}(V_0)$, then $\sigma\pi(x) \in \sigma\pi(\sigma\pi)^{-1}(V_0) = V_0$. Therefore, the point $\sigma\pi(x)$ is distinct from $x$, since $x \in W_0 \subset U_0$, and $U_0$ and $V_0$ are disjoint.
\end{proof}

Every closed subset of $X$  touching each orbit in exactly one point defines a cross section of $\pi:X\to X|G$. More precisely, the next proposition is true.

\begin{proposition}\label{prop-3}
Let $G$ be a compact group, $X$ be a distributive binary $G$-space, and $A$ be a closed subset of the space $X$ that intersects with each orbit $G(x,x)$ at exactly one point. Then the map $\sigma: X|G \rightarrow X$, defined by the formula $\sigma(x^*) = A \cap G(x,x)$, is a cross section of the projection $\pi: X \rightarrow X|G$.
\end{proposition}

\begin{proof}
Let $C$ be an arbitrary closed subset of $A$. By Theorem \ref{th-2}, the set $\sigma^{-1}(C) = \pi(C)$ is closed in $X|G$. Therefore, the map $\sigma: X|G \rightarrow X$ is continuous. 

The condition $\pi\sigma = 1_{X|G}$ immediately follows from the definition of $\sigma$.
\end{proof}

Propositions \ref{prop-2} and \ref{prop-3} demonstrate that there exists a bijective correspondence between the cross sections and the closed sets of $X$ that intersect with each orbit at exactly one point. Based on this, by <<cross section>> we will also understand a closed set that is the image of some cross section.

\begin{theorem}\label{th-4}
Let $G$ be a compact topological group, $X$ and $Y$ distributive binary $G$-spaces, and $A$ a section of the projection $\pi:X\to X|G$. Suppose $f:A\to Y$ is a continuous map such that 
\begin{multline*}
(*) \quad g(a,a)=h(k(a',a'),s(a'',a'')) \ 
implies \\
 g(f(a),f(a))=h(k(f(a'),f(a')),s(f(a''),f(a''))),
\end{multline*}
where $a,a',a''\in A$ and $g,h,k,s\in G$. Then $f$ has a unique continuous bi-equivariant extension $\widetilde{f} : X \to Y$.
\end{theorem}

\begin{proof}
Let \(x \in X\) be an arbitrary point. Suppose that \(x\) belongs to the orbit of the point \(a \in A\). This means that there exists an element \(g \in G\) such that \(x = g(a,a)\). Now, let's define the unique possible bi-equivariant extension \(\tilde{f}: X \rightarrow Y\) of the map \(f: A \rightarrow Y\) by the formula
$$\tilde{f}(x) = g(f(a),f(a)).$$

Note that the map $\tilde{f}$ is defined correctly because if \(g(a,a) = h(a,a)\), then \(g(f(a),f(a)) = h(f(a),f(a))\) due to condition (*).

The map $\tilde{f}$ is an extension of the map $f$ since 
$$\tilde{f}(a) = \tilde{f}(e(a,a)) = e(f(a),f(a)) = f(a)$$ 
for any point $a \in A$.

Now, let's prove that $\tilde{f}:X\to Y$ is a bi-equivariant map, i.e., for arbitrary $x, x' \in X$ and $g \in G$, the equality
$$\tilde{f}(g(x,x')) = g(\tilde{f}(x), \tilde{f}(x'))$$
holds.

Let  
$$x=k(a,a), \ x' = s(a',a') \ \text{and} \ g(x,x')= h(\bar{a},\bar{a}),$$ 
where $a, a', \bar{a} \in A$, $k,h,s \in G$. Then 
$$h(\bar{a},\bar{a})  = g(k(a,a),s(a',a')).$$ 
Consequently, due to condition (*), we have:
$$h(f(\bar{a}),f(\bar{a}))  = g(k(f(a),f(a)),s(f(a'),(a'))).$$ 
This means,
\begin{multline*} 
\tilde{f}(g(x,x'))=\tilde{f}(h(\bar{a},\bar{a}))=h(f(\bar{a}), f(\bar{a})) = g(k(f(a),f(a)),s(f(a'),(a'))) = \\
= g(\tilde{f}(k(a,a)), \tilde{f}(s(a',a'))) = g(\tilde{f}(x), \tilde{f}(x')).
\end{multline*}

The continuity of the map $\tilde{f}$ follows from the continuity of the map $f$, the cross section of the projection $\pi: X\to X|G$, and the binary actions of the group $G$ on the spaces $X$ and $Y$.
\end{proof}

Now suppose that $X$ is a binary $G$-space and let $x, x'\in X$. It is easy to note that the set
$$G_{(x,x')}=\{g\in G, \ g(x,x')=x'\}$$
is a closed subgroup of the group $G$. This subgroup is called the {\it isotropy group} (or {\it stability group}) of the point $x'$  with respect to $x$ or the {\it isotropy group} of the pair $(x,x')$.

It can be proven that the following equality holds:
$$G_{(x, g(x,x'))}= gG_{(x,x')}g^{-1}$$
for any $g\in G$ and $x,x'\in X$.

We have the following necessary condition for the bi-equivariance of a map $f:X\to Y$, which has an elementary proof.

\begin{proposition}
Let $f:X\to Y$ be a bi-equivariant map between binary $G$-spaces $X$ and $Y$. Then
\begin{equation*}
G_{(x,x')}\subset G_{(f(x),f(x'))}
\end{equation*}
for all $x,x'\in X$.
\end{proposition}

The next result follows from Theorem \ref{th-4}.

\begin{corollary}\label{cor-1}
Let $G$ be a compact group, $X$ and $Y$ be distributive binary $G$-spaces, $A$ be a cross section of the projection $\pi:X\to X|G$, and $f:A\to Y$ be a continuous map satisfying the condition
\begin{equation}\label{eq-sub}
G_{(a,a')}\subset G_{(f(a),f(a'))}
\end{equation}
for any $a, a'\in A$. Then the map $f$ has a unique continuous bi-equivariant extension $\tilde{f}:X\to Y$.
\end{corollary}

\begin{proof}
It should be noted that condition $(*)$ of Theorem \ref{th-4} follows from \eqref{eq-sub}. Indeed, if  $g(a,a') = h(a, a')$, then $g^{-1}h\in G_{a,a'}\subset G_{(f(a),f(a'))}$, which implies that $g(f(a),f(a')) = h(f(a), f(a'))$.
\end{proof}


\end{document}